%% file: ec-isogeny.tex
\documentclass[preprint,1p]{elsarticle}

\usepackage[english]{babel}
\usepackage[utf8]{inputenc}
\usepackage{bbm}
\usepackage{amsmath}
\usepackage{amsthm}
\usepackage{amssymb}
\usepackage{mathrsfs}
\usepackage{mdwlist}
\usepackage{graphicx}
\usepackage{mathrsfs}
\usepackage{url}
\usepackage{color}
\usepackage[all]{xypic}

\renewcommand{\le}{\leqslant}
\renewcommand{\ge}{\geqslant}  
\newcommand{\clot}[1]{\bar{#1}}  
\newcommand{\card}[1]{\# #1}  
\newcommand{\N}{\mathbb{N}}  
\newcommand{\Z}{\mathbb{Z}}  
\newcommand{\K}{\mathbb{K}}  
\newcommand{\U}{\mathbb{U}}  
\newcommand{\F}{\mathbb{F}}  
\newcommand{\Q}{\mathbb{Q}}  
\newcommand{\isom}{\cong}  
\newcommand{\frob}{\varphi}  
\DeclareMathOperator{\Gal}{Gal}  
\newcommand{\euler}{\phi}  
\DeclareMathOperator{\ord}{ord}  
\newcommand{\0}{\mathcal{O}}  
\newcommand{\isog}[1]{\mathcal{#1}}  
\newcommand{\I}{\isog{I}}  
\newcommand{\frobisog}{\phi}  
\newcommand{\tildO}{\tilde{O}}  
\newcommand{\Mult}{\mathrm{\sf M}}  
\newcommand{\ModComp}{\mathrm{\sf C}}  
\newcommand{\alg}[1]{{\sf #1}}  

\newtheorem{definition}{Definition}
\newtheorem{theorem}{Theorem}

\newtheorem{proposition}[definition]{Proposition}

\journal{Journal of Number Theory}

\begin{document}

\begin{frontmatter}

\title{Fast algorithms for computing isogenies between ordinary
  elliptic curves in small characteristic}
\author{Luca De Feo}
\ead{luca.defeo@polytechnique.edu}
\address{LIX, {\'E}cole Polytechnique, 91128 Palaiseau, France}

\begin{abstract}
  The problem of computing an explicit isogeny between two given
  elliptic curves over $\F_q$, originally motivated by point counting,
  has recently awaken new interest in the cryptology community thanks
  to the works of Teske and Rostovstev \& Stolbunov.

  While the large characteristic case is well understood, only
  suboptimal algorithms are known in small characteristic; they are
  due to Couveignes, Lercier, Lercier \& Joux and Lercier \& Sirvent.
  In this paper we discuss the differences between them and run some
  comparative experiments. We also present the first complete
  implementation of Couveignes' second algorithm and present
  improvements that make it the algorithm having the best asymptotic
  complexity in the degree of the isogeny.
\end{abstract}


\end{frontmatter}

\input{introduction}
\input{preliminaries}
\input{C2}
\input{C2-AS}
\input{C2-AS-FI}
\input{C2-AS-FI-MC}
\input{implementation}
\input{benchmarks}

\section*{Acknowledgements}
We would like to thank J.-M.~Couveignes, F.~Morain, E.~Schost and
B.~Smith for useful discussions and precious proof-reading.

\input{biblio}
\end{document}

%% file: introduction.tex
\section{Introduction}

The problem of computing an explicit degree $\ell$ isogeny between two
given elliptic curves over $\F_q$ was originally motivated by point
counting methods based on Schoof's algorithm \cite{Atk91},
\cite{Elk91}, \cite{Sch95}. A review of the most efficient algorithms
to solve this problem is given in \cite{BoMoSaSc08} together with a
new quasi-optimal algorithm; however, all the algorithms presented in
\cite{BoMoSaSc08} are limited to the case $\ell\ll p$ where $p$ is the
characteristic of $\F_q$. This is satisfactory for cryptographic
applications where one takes $p=q$ or $p=2$; indeed in the former case
Schoof's algorithm needs $\ell\in O(\log p)$, while in the latter case
there's no need to compute explicit isogenies since $p$-adic methods
based on \cite{Sat00} are preferred to Schoof's algorithm.

Nevertheless, the problem of computing explicit isogenies in the case
where $p$ is small compared to $\ell$ remains of theoretical interest
and can find practical applications in newer cryptosystems such as
\cite{Tes06}, \cite{RoSt06}. The first algorithm to solve this problem
was given by Couveignes and made use of formal groups \cite{Cou94}; it
takes $\tildO(\ell^3\log q)$ operations in $\F_p$ assuming $p$ is
constant, however it has an exponential complexity in $\log
p$. Another algorithm by Lercier specific to $p=2$ uses some linear
properties of the problem to build a linear system from whose solution
the isogeny can be deduced \cite{Ler96}; its complexity is conjectured
to be $\tildO(\ell^3\log q)$ operations in $\F_p$, but it has a much
better constant factor than \cite{Cou94}. At the moment we write, the
latter algorithm is by many orders of magnitude the fastest algorithm
to solve practical instances of the problem when $p=2$, thus being the
\emph{de facto} standard for cryptographic use.

$p$-adic methods were used by Joux and Lercier \cite{JL06} and Lercier
and Sirvent \cite{LeSi09} to solve the isogeny problem. The former
method has complexity $\tildO(\ell^2(1 + \ell/p)\log q)$ operations in
$\F_p$, which makes it well adapted to the case where $p\sim\log
q$. The latter has complexity $\tildO(\ell^3 + \ell\log q^2)$
operations in $\F_p$, making it the best algorithm to our knowledge
for the case where $p$ is not constant.

\paragraph{The algorithm C2 and its variants}
Finally, the algorithm having the best asymptotic complexity in $\ell$
was proposed again by Couveignes in \cite{Cou96}; we will refer to
this original version as ``C2''\footnote{As opposed to the algorithm
  presented in \cite{Cou94}, an algorithm ``C2'' shares many
  similarities with.}. Its complexity --supposing $p$ is fixed-- was
estimated in \cite{Cou96} as being $\tildO(\ell^2\log q)$ operations
in $\F_p$, but with a precomputation step requiring $\tildO(\ell^3\log
q)$ operations and large memory requirements. However, some more work
is needed to effectively reach these bounds, while a straightforward
implementation of C2 has an overall asymptotic complexity of
$\tildO(\ell^3\log q)$ operations, as we will argue in Section
\ref{sec:C2}.

Subsequent work by Couveignes \cite{Cou00}, and more recently
\cite{DFS09}, use Artin-Schreier theory to avoid the precomputation
step of C2 and drop the memory requirements to $\tildO(\ell\log q +
\log^2 q)$ elements of $\F_p$. However, this is still not enough to
reduce the overall complexity of the algorithm, as we will argue in
Section \ref{sec:C2-AS}. We refer to this variant as ``C2-AS''.

In the present paper we give a complete review of Couveignes'
algorithm, we present new variants that reach the foreseen quadratic
bound in $\ell^2$ and prove an accurate complexity estimate which
doesn't suppose $p$ to be fixed. We also run experiments to compare
the performances of C2 with other algorithms.

\paragraph{Notation and plan}
In the rest of the paper $p$ is a prime, $d$ a positive integer,
$q=p^d$ and $\F_q$ is the field with $q$ elements. For an elliptic
curve $E$ and a field $\K$ embedded in an algebraic closure
$\clot{\K}$, we note by $E(\K)$ the set of $\K$-rational points and by
$E[m]$ the $m$-torsion subgroup of $E(\clot{\K})$. The group law on
the elliptic curve is noted additively, its zero is the point at
infinity, noted $\0$. For an affine point $P$ we note by $x(P)$ its
abscissa and by $y(P)$ its ordinate. We will restrict ourselves to the
case of ordinary elliptic curves, thus $E[p^k]\isom\Z/p^k\Z$.

Unless otherwise stated, all time complexities will be measured in
number of operations in $\F_p$ and all space complexities in number of
elements of $\F_p$; we do not assume $p$ to be constant. We use the
$O$, $\Theta$ and $\Omega$ notations to state respectively upper
bounds, tight bounds and lower bounds for asymptotic complexities. We
also use the notation $\tildO_x$ that forgets polylogarithmic factors
in the variable $x$, thus $O(xy\log x \log y)\subset\tildO_x(xy\log
y)\subset\tildO_{x,y}(xy)$. We simply note $\tildO$ when the variables
are clear from the context.

We let $2<\omega\le3$ be the exponent of linear algebra, that is an
integer such that $n\times n$ matrices can be multiplied in $n^\omega$
operations. We let $\Mult:\N\rightarrow\N$ be a \emph{multiplication
  function}, such that polynomials of degree at most $n$ with
coefficients in $\F_p$ can be multiplied in $\Mult(n)$ operations,
under the conditions of \cite[Ch. 8.3]{vzGG}. Typical orders of
magnitude are $O(n^{\log_23})$ for Karatsuba multiplication or
$O(n\log n\log\log n)$ for FFT multiplication. Similarly, we let
$\ModComp:\N\rightarrow\N$ be the complexity of \emph{modular
  composition}, that is a function such that $\ModComp(n)$ is the
number of field operations needed to compute $f\circ g\bmod h$ for
$f,g,h\in\K[X]$ of degree at most $n$ with coefficients in an
arbitrary field $\K$. The best known algorithm is~\cite{BrKu78}, this
implies $\ModComp(n)\in O\left(n^{\frac{\omega+1}{2}}\right)$. Note
that in a boolean RAM model, the algorithm of~\cite{KeUm08} takes
quasi-linear time.

\paragraph{Organisation of the paper}
In Section~\ref{sec:preliminaries} we give preliminaries on elliptic
curves and isogenies. In Sections~\ref{sec:C2}
through~\ref{sec:C2-AS-FI-MC} we develop the algorithm C2 and we
incrementally improve it by giving a new faster variant in each
Section. Section~\ref{sec:implementation} gives technical details on
our implementations of the algorithms of this paper and
of~\cite{LeSi09}. Finally in Section~\ref{sec:benchmarks} we comment
the results of the experiments we ran on our implementations.

%

%% file: preliminaries.tex
\section{Preliminaries on Isogenies}
\label{sec:preliminaries}

Let $E$ be an ordinary elliptic curve over the field $\F_q$. We
suppose it is given to us as the locus of zeroes of an affine
Weierstrass equation
\[y^2 + a_1xy + a_3y = x^3 + a_2x^2 + a_4x + a_6
\qquad a_1,\ldots,a_6\in\F_q\text{.}\]

\paragraph{Simplified forms} If $p>3$ it is well known that the curve
$E$ is isomorphic to a curve in the form
\begin{equation}
  \label{eq:weierstrass>3}
  y^2 = x^3 + ax + b
\end{equation}
and its $j$-invariant is $j(E) = \frac{1728(4a)^3}{16(4a^3 + 27b^2)}$.

When $p=3$, since $E$ is ordinary, it is isomorphic to a curve
\begin{equation}
  \label{eq:weierstrass=3}
  y^2 = x^3 + ax^2 + b
\end{equation}
and its $j$-invariant is $j(E) = -\frac{a^3}{b}$.

Finally, when $p=2$, since $E$ is ordinary, it is isomorphic to a curve
\begin{equation}
  \label{eq:weierstrass=2}
  y^2 + xy = x^3 + ax^2 + b
\end{equation}
and its $j$-invariant is $j(E) = \frac{1}{b}$.

These isomorphism are easy to compute and we will always assume that
the elliptic curves given to our algorithms are in such simplified
forms.

\paragraph{Isogenies}
Elliptic curves are endowed with the classic group structure through
the chord-tangent law. A group morphism having finite kernel is called
an \emph{isogeny}. Isogenies are regular maps, as such they can be
represented by rational functions. An isogeny is said to be
$\K$-rational if it is $\K$-rational as regular map; its degree is the
degree of the regular map.

One important property about isogenies is that they factor the
multiplication-by-$m$ map.

\begin{definition}[Dual isogeny]
  Let $\I : E \rightarrow E'$ be a degree $m$ isogeny. There exists an
  unique isogeny $\hat{\I} : E' \rightarrow E$, called the \emph{dual
    isogeny} such that
  \[\I\circ\hat{\I} = [m]_E \qquad\text{and}\qquad \hat{\I}\circ\I =
  [m]_{E'}\]
\end{definition}

As regular maps, isogenies can be separable, inseparable or purely
inseparable. In the case of finite fields, purely inseparable
isogenies are easily understood as powers of the frobenius map. Let
\[E^{(p)} : y^2 + a_1^pxy + a_3^py = x^3 + a_2^px^2 + a_4^px + a_6^p\]
then the map
\begin{align*}
  \frobisog : E &\rightarrow E^{(p)}\\
          (x,y) &\mapsto (x^p,y^p)
\end{align*}
is a degree $p$ purely inseparable isogeny. Any purely inseparable
isogeny is a composition of such frobenius isogenies.

Let $E$ and $E'$ be two elliptic curves defined over $\F_q$, by
finding an \emph{explicit isogeny} we mean to find an
($\F_q$-rational) rational function from $E(\clot{\F}_q)$ to
$E'(\clot{\F}_q)$ such that the map it defines is an isogeny.

\begin{figure}
  \centering
  \[\xymatrix{
    E \ar[r]^{[m]}\ar@/_1pc/[rrr]_{\I'} & E \ar[r]^\I & E' \ar[r]^{\frobisog^n} & E'^{(p^n)}\\
  }\]
  \label{fig:fact}
  \caption{Factorization of an isogeny. $\I'$ has kernel $E[m]\oplus\ker\I$.}
\end{figure}

Since an isogeny can be uniquely factored in the product of a
separable and a purely inseparable isogeny, we focus ourselves on the
problem of computing explicit separable isogenies. Furthermore one can
factor out multiplication-by-$m$ maps, thus reducing the problem to
compute explicit separable isogenies with cyclic kernel (see figure
\ref{fig:fact}).

In the rest of this paper, unless otherwise stated, by $\ell$-isogeny
we mean a separable isogeny with kernel isomorphic to $\Z/\ell\Z$.

\paragraph{Vélu formulae}
For any finite subgroup $G \subset E(\clot{\K})$, Vélu formulae
\cite{Vel71} give in a canonical way an elliptic curve $\bar{E}$ and
an explicit isogeny $\I:E\rightarrow \bar{E}$ such that
$\ker\I=G$. The isogeny is $\K$-rational if and only if the polynomial
vanishing on the abscissae of $G$ belongs to $\K[X]$.

In practice, if $E$ is defined over $\F_q$ and if
\[h(X) = \prod_{\substack{P\in G\\P\ne\0}}(X - x(P)) \in \F_q[X]\]
is known, Vélu formulae compute a rational function
\begin{equation}
  \label{eq:isog}
  \bar{\I}(x,y) = \left(\frac{g(x)}{h(x)}, \frac{k(x,y)}{l(x)}\right)  
\end{equation}
and a curve $\bar{E}$ such that $\bar{\I} : E\rightarrow\bar{E}$ is an
$\F_q$-rational isogeny of kernel $G$. A consequence of Vélu formulae
is
\begin{equation}
  \label{eq:velu-deg}
  \deg g = \deg h + 1 = \card{G}
  \text{.}
\end{equation}

Given two curves $E$ and $E'$, Vélu formulae reduce the problem of
finding an explicit isogeny between $E$ and $E'$ to that of finding
the kernel of an isogeny between them. Once the polynomial $h(X)$
vanishing on $\ker\I$ is found, the explicit isogeny is computed
composing Vélu formulae with the isomorphism between $\bar{E}$ and
$E'$ as in figure \ref{fig:velu}.

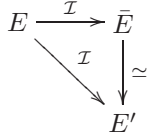
\begin{figure}
  \centering
  \[\xymatrix{
    E \ar[r]^{\bar{\I}} \ar[rd]^\I & \bar{E} \ar[d]^{\simeq}\\
    & E'
  }\]
  \caption{Using Vélu formulae to compute an explicit isogeny.}
  \label{fig:velu}
\end{figure}

%

%% file: C2.tex
\section{The algorithm C2}
\label{sec:C2}

The algorithm we refer to as C2 was originally proposed in
\cite{Cou96}. It takes as input two elliptic curves $E, E'$ and an
integer $\ell$ prime to $p$ and it returns, if it exists, an
$\F_q$-rational isogeny of degree $\ell$ between $E$ and $E'$. It only
works in odd characteristic.

\subsection{The original algorithm}
Suppose there exists an $\F_q$-rational isogeny
$\I:E\rightarrow E'$ of degree $\ell$. Since $\ell$ is prime to $p$
one has $\I(E[p^k]) = E'[p^k]$ for any $k$.

Recall that $E[p^k]$ and $E'[p^k]$ are cyclic groups. C2 iteratively
computes generators $P_k,P_k'$ of $E[p^k]$ and $E'[p^k]$
respectively. Now C2 makes the guess $\I(P_k) = P_k'$; then, if $\I$
is given by rational fractions as in \eqref{eq:isog},
\begin{equation}
  \label{eq:C2:I}
  \frac{g\bigl(x([i]P_k)\bigr)}{h\bigl(x([i]P_k)\bigr)} = x([i]P_k')
  \quad\text{for $i\in\Z/p^k\Z$} 
\end{equation}
and by \eqref{eq:velu-deg} $\deg g = \deg h + 1 = \ell$.

Using \eqref{eq:C2:I} one can compute the rational fraction
$\frac{g(X)}{h(X)}$ through Cauchy interpolation over the points of
$E[p^k]$ for $k$ large enough. C2 takes $p^k > 4\ell - 2$,
interpolates the rational fraction and then checks that it corresponds
to the restriction of an isogeny to the $x$-axis. If this is the case,
the whole isogeny is computed through Vélu formulae and the algorithm
terminates. Otherwise the guess $\I(P_k) = P_k'$ was wrong, then C2
computes a new generator for $E'[p^k]$ and starts over again.

We now go through the details of the algorithm.

\paragraph{The $p$-torsion}
The computation of the $p$-torsion points follows from the work of
Gunji \cite{Gun76}. Here we suppose $p\ne2$.

\begin{definition}
  \label{def:hasse}
  Let $E$ have equation $y^2 = f(x)$. The \emph{Hasse invariant} of
  $E$, noted $H_E$, is the coefficient of $X^{p-1}$ in
  $f(X)^{\frac{p-1}{2}}$.
\end{definition}

Gunji shows the following proposition and gives formulae to compute
the $p$-torsion points.

\begin{proposition}
  \label{th:gunji}
  Let $c=\sqrt[p-1]{H_E}$; then, the $p$-torsion points of $E$ are
  defined in $\F_q[c]$ and their abscissae are defined in $\F_q[c^2]$.
\end{proposition}

\paragraph{The $p^k$-torsion}
$p^k$-torsion points are iteratively computed via $p$-descent. The
basic idea is to split the multiplication map as $[p] = \frobisog\circ
V$ and invert each of the components. The purely inseparable isogeny
$\frobisog$ is just a frobenius map and the separable isogeny $V$ can
be computed by Vélu formulae once the $p$-torsion points are
known. Although this is reasonably efficient, pulling $V$ back may
involve factoring polynomials of degree $p$ in some extension field.

A finer way to do the $p$-descent, as suggested in the original paper
\cite{Cou96}, is to use the work of Voloch \cite{Vol90}. Suppose
$p\ne2$, let $E$ and $\widetilde{E}$ have equations respectively
\begin{align*}
  y^2&=f(x)=x^3+a_2x^2+a_4x+a_6 \;\text{,}\\
  \tilde{y}^2&=\tilde{f}(\tilde{x}) = \tilde{x}^3 +
  \sqrt[p]{a_2}\tilde{x}^2 + \sqrt[p]{a_4}\tilde{x} + \sqrt[p]{a_6}
  \;\text{,}
\end{align*}
set
 \begin{equation}
  \label{eq:voloch:cover}
  \tilde{f}(X)^{\frac{p-1}{2}} = \alpha(X) + H_{\widetilde{E}}X^{p-1} + X^p\beta(X)
\end{equation}
with $\deg \alpha < p-1$ and $H_{\widetilde{E}}$ the Hasse invariant
of $\widetilde{E}$. Voloch shows the following proposition.

\begin{proposition}
  \label{th:voloch}
  Let $\tilde{c} = \sqrt[p-1]{H_{\widetilde{E}}}$, the cover of
  $\widetilde{E}$ defined by
  \begin{equation}
    \label{th:voloch:cover}
    C:\; \tilde{z}^p - \tilde{z} = \frac{\tilde{y}\beta(\tilde{x})}{\tilde{c}^p}
  \end{equation}
  is an étale cover of degree $p$ and is isomorphic to $E$ over
  $\F_q[\tilde{c}]$; the isomorphism is given by
  \begin{equation}
    \label{th:voloch:isom}
    \left\{
      \begin{aligned}
        (\tilde{x}, \tilde{y}) &= V(x, y)\\
        \tilde{z} &= -\frac{y}{\tilde{c}^p}\sum_{i=1}^{p-1}\frac{1}{x - x([i]P_1)}
      \end{aligned}
    \right.
  \end{equation}
  where $P_1$ is a primitive $p$-torsion point of $E$.
\end{proposition}

The descent is then performed as follows: starting from a point $P$ on
$E$, first pull it back along $\frobisog$, then take one of its
pre-images in $C$ by solving equation \eqref{th:voloch:cover}, finally
use equation \eqref{th:voloch:isom} to land on a point $P'$ in $E$.
The proposition guarantees that $[p]P' = P$. The descent is pictured
in figure \ref{fig:voloch}.

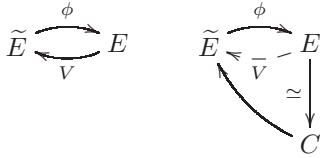
\begin{figure}
  \centering
  \[
  \xymatrix{\widetilde{E}\ar@/^/[r]^{\frobisog} & E\ar@/^/[l]^{V}}
  \qquad
  \xymatrix{
    \widetilde{E}\ar@/^/[r]^{\frobisog} & E\ar@/^/@{-->}[l]^{V}\ar[d]_{\simeq}\\
    & C\ar@/^/[ul]
  }
  \]
  
  \caption{Two ways of doing the $p$-descent: standard on the left and via a degree $p$ cover on the right}
  \label{fig:voloch}
\end{figure}

The reason why this is more efficient than a standard descent is the
shape of equation \eqref{th:voloch:cover}: it is an Artin-Schreier
equation and it can be solved by many techniques, the simplest being
linear algebra (as was suggested in \cite{Cou96}). Once a solution
$\tilde{z}$ to \eqref{th:voloch:cover} is known, solving in $x$ and $y$ the
bivariate polynomial system \eqref{th:voloch:isom} takes just a GCD
computation (explicit formulae were given by Lercier in
\cite[$\S$6.2]{Ler97}, we give some slightly improved ones in Section
\ref{sec:implementation}). Compare this with a generic factoring
algorithm needed by standard descent.

Solving Artin-Schreier equations is the most delicate task of the
descent and we will further discuss it.

\paragraph{Cauchy interpolation}
Interpolation reconstructs a polynomial from the values it takes on
some points; Cauchy interpolation reconstructs a rational
fraction. The Cauchy interpolation algorithm is divided in two phases:
first find the polynomial $P$ interpolating the evaluation points,
then use rational fraction reconstruction to find a rational fraction
congruent to $P$ modulo the polynomial vanishing on the points. The
first phase is carried out through any classical interpolation
algorithm, while the second is similar to an XGCD computation. See
\cite[$\S$5.8]{vzGG} for details.

Cauchy interpolation needs $n+2$ points to reconstruct a degree
$(k,n-k)$ rational fraction. This, together with \eqref{eq:velu-deg},
justifies the choice of $k$ such that $p^k > 4\ell - 2$. Some of our
variants of C2 will interpolate only on the primitive $p^k$-torsion
points, thus requiring the slightly larger bound $\euler(p^k) \ge
4\ell - 2$. This is not very important to our asymptotical analysis
since in both cases $p^k \in O(\ell)$.

\paragraph{Recognising the isogeny}
Once the rational fraction $\frac{g(X)}{h(X)}$ has been computed, one
has to verify that it is indeed an isogeny. The first test is to check
that the degrees of $g$ and $h$ match equation \eqref{eq:velu-deg}, if
they don't, the equation can be discarded right away and the algorithm
can go on with the next trial. Next, one can check that $h$ is indeed
the square of a polynomial (or, if $\ell$ is even, the product of one
factor of the $2$-division polynomial and a square polynomial). This
two tests are usually enough to detect an isogeny, but, should they
lie, one can still check that the resulting rational function is
indeed a group morphism by trying some random points on $E$.

\subsection{The case $p=2$}
\label{sec:p=2}
The algorithm as we have presented it only works when $p\ne2$, it is
however an easy matter to generalise it. The only phase that doesn't
work is the computation of the $p^k$-torsion points. For curves in the
form \eqref{eq:weierstrass=2} the only $2$-torsion point is
$(0,\sqrt{b})$.

Voloch formulae are hard to adapt, nevertheless a $2$-descent on the
Kummer surface of $E$ can easily be performed since the doubling
formula reads
\begin{equation}
  x([2]P) = \frac{b}{x(P)^2} + x(P)^2 =
  \frobisog\left(\frac{\sqrt{b} + x(P)^2}{x(P)} \right) = \frobisog\circ V
  \;\text{.}
\end{equation}
Given point $x_P$ on $K_E$, a pull-back along $\frobisog$ gives a
point $\tilde{x}_P$ on $K_{\widetilde{E}}$. Then pulling $V$ back
amounts to solve
\begin{equation}
  \label{eq:2-descent}
  x^2 + \tilde{x}_Px = \sqrt{b}
\end{equation}
and this can be turned in an Artin-Schreier equation through the
change of variables $x \rightarrow x'\tilde{x}_P$.

From the descent on the Kummer surfaces one could deduce a full
$2$-descent on the curves by solving a quadratic equation at each step
in order recover the $y$ coordinate, but this would be too
expensive. Fortunately, the $y$ coordinates are not needed by the
subsequent steps of the algorithm, thus one may simply ignore
them. Observe in fact that even if $K_E$ does not have a group law,
the restriction of scalar multiplication is well defined and can be
computed through Montgomery formulae \cite{Mon87}. This is enough to
compute all the abscissae of the points in $E[p^k]$ once a generator
is known.

\subsection{Complexity analysis}
\label{sec:C2:complexity}
Analysing the complexity of C2 is a delicate matter since the
algorithm relies on some black-box computer algebra algorithms in
order to deal with finite extensions of $\F_q$. The choice of the
actual algorithms may strongly influence the overall complexity of C2.
In this section we will only give some lower bounds on the complexity
of C2, since a much more accurate complexity analysis will be carried
out in Section \ref{sec:C2-AS}.

\paragraph{$p$-torsion}
Applying Gunji formulae first requires to find $c$ and $c'$, $p-1$-th
roots of $H_E$ and $H_{E'}$, and build the field extension $\F_q[c] =
\F_q[c']$. Independently of the actual algorithm used, observe that in
the worst case $\F_q[c]$ is a degree $p-1$ extension of $\F_q$, thus
simply representing one of its elements requires $\Theta(pd)$ elements
of $\F_p$.

Subsequently, the main cost in Gunji's formulae is the computation of
the determinant of a $\frac{p-1}{2}\times\frac{p-1}{2}$
quadri-diagonal matrix (see \cite{Gun76}). This takes $\Theta(p^2)$
operations in $\F_q[c]$ by Gauss elimination, that is no less than
$\Omega(p^3d)$ operations in $\F_p$.

\paragraph{$p^k$-torsion}
During the $p$-descent, factoring of equations \eqref{th:voloch:cover}
or \eqref{eq:2-descent} may introduce some field extensions over
$\F_q[c]$. Observe that an Artin-Schreier polynomial is either
irreducible or totally split, so at each step of the $p$-descent we
either stay in the same field or we take a degree $p$ extension. This
shows that in the worst case, we have to take an extension of degree
$p^{k-1}$ over $F_q[c]$. The following proposition, which is a
generalisation of \cite[Prop. 26]{Ler97}, states precisely how likely
this case is.

\begin{proposition}
  \label{th:tower}
  Let $E$ be an elliptic curve over $\F_q$, we note $\U_i$ the
  smallest field extension of $\F_q$ such that $E[p^i]\subset
  E(\U_i)$. For any $i\ge1$, either $[\U_{i+1}:\U_i] = p$ or
  $\U_{i+1}=\U_i=\cdots=\U_1$.
\end{proposition}
\begin{proof}
  Observe that the action of the Frobenius $\frobisog$ on $E[p]$ is
  just multiplication by the trace $t$, in fact the equation
  \[\frobisog^2 - [t \bmod p]\circ\frobisog + [q \bmod p] = 0\]
  has two solutions, namely $[t \bmod p]$ and $[0 \bmod p]$, but the
  second can be discarded since it would imply that $\frobisog$ has
  non-trivial kernel.  By lifting this solution, one sees that the
  action of $\frobisog$ on the Tate module $\mathcal{T}_p(E)$ is equal
  to multiplication by some $\tau\in\Z_p$.

  Note $G$ the absolute Galois group of $\F_q$, there is a well known
  action of $G$ on $\mathcal{T}_p(E)$. Since $G$ is generated by the
  Frobenius automorphism of $\F_q$, the restriction of this action to
  $E[p^k]$ is equal to the action (via multiplication) of the subgroup
  of $(\Z/p^k\Z)^\ast$ generated by $\tau_k = \tau \bmod p^k$. Hence
  $[\U_k:\F_q] = \ord(\tau_k)$.

  Then, for any $k>1$, \cite[Corollary 4]{Ler97} applied to
  $\tau_{k+1}=\tau\bmod p^{k+1}$ shows that
  $\ord(\tau_{k+1})=\ord(\tau_k)$ implies
  $\ord(\tau_k)=\ord(\tau_{k-1})$ and this concludes the proof.
\end{proof}

Thus for any elliptic curve there is an $i_0$ such that $[\U_i:\U_1] =
p^{i-i_0}$ for any $i \ge i_0$. This shows that the worst and the
average case coincide since for any fixed curve $[\U_k:\U_1] \in
\Theta(p^k)$ asymptotically. In this situation, one needs
$\Theta(p^kd)$ elements of $\F_p$ to store an element of $\U_k$.

Now the last iteration of the $p$-descent needs to solve an
Artin-Schreier equation in $\U_k$. To do this C2 precomputes the
matrix of the $\F_q$-linear application $(x^q-x):\U_k\rightarrow\U_k$
and its inverse, plus the matrix of the $\F_p$-linear application
$(x^p-x):\F_q\rightarrow\F_q$ and its inverse. The former is the most
expensive one and takes $\Theta(p^{\omega k})$ operations in $\F_q$,
that is $\Omega(p^{\omega k}d) = \Omega(\ell^\omega d)$ operations in
$\F_p$, plus a storage of $\Theta(\ell^2d)$ elements of
$\F_p$. Observe that this precomputation may be used to compute any
other isogeny with domain $E$.

After the precomputation has been done, C2 successively applies the
two inverse matrices; details can be found in
\cite[$\S$2.4]{Cou96}. This costs at least $\Omega(\ell^2d)$.

\paragraph{Interpolation}
The most expensive part of Cauchy interpolation is the polynomial
interpolation phase. In fact, simply representing a polynomial of
degree $p^k-1$ in $\U_k[X]$ takes $\Theta(p^{2k}d)$ elements, thus at
least $\Omega(\ell^2d)$ operations are needed to interpolate unless
special care is taken. This contribution due to arithmetics in $\U_k$
had been underestimated in the complexity analysis of \cite{Cou96},
which gave an estimate of $\Omega(\ell d\log\ell)$ operations for this
phase. We will give more details on interpolation in Section
\ref{sec:C2-AS-FI}.

\paragraph{Recognising the isogeny}
The cost of testing for squareness of the denominator and other tests
is negligible compared to the rest of the algorithm. Nevertheless it
is important to realize that on average half of the $\euler(p^k)$
mappings from $E[p^k]$ to $E'[p^k]$ must be tried before finding the
isogeny, for only one of these mappings corresponds to it. This
implies that the Cauchy interpolation step must be repeated an average
of $\Theta(p^k)$ times, thus contributing a $\Omega(\ell^3d)$ to the
total complexity.

Summing up all the contributions one ends up with the following lower
bound
\begin{equation}
  \label{eq:C2:complexity}
  \Omega(\ell^3d + p^3d)
\end{equation}
plus a precomputation step whose cost is negligible compared to this
one and a space requirement of $\Theta(\ell^2d)$ elements. In the next
sections we will see how to make all these costs drop.




%

%% file: C2-AS.tex
\section{The algorithm C2-AS}
\label{sec:C2-AS}

One of the most expensive steps of C2 is the resolution of an
Artin-Schreier equation in an extension field $\U_i$. In \cite{Cou00}
Couveignes gives an approach alternative to linear algebra to solve
this problem. First it builds the whole tower $(\U_1=\F_q[c], \ldots,
\U_k)$ of intermediate extensions, then it solves an Artin-Schreier
equation in $\U_i$ recursively by reducing it to another
Artin-Schreier equation in $\U_i$. Details are in \cite{Cou00} and
\cite{DFS09}.

To solve the final Artin-Schreier equation in $\U_1=\F_q[c]$ one
resorts to linear algebra, thus precomputing the inverse matrix of the
$\F_p$-linear application $(x^p-x):\U_1\rightarrow\U_1$.

\subsection{Complexity analysis}
\label{sec:C2-AS:complexity}
How effective this method is depends on the way algebra is performed
in the tower $(\U_1,\ldots,\U_k)$. The present author and Schost
\cite{DFS09} recently presented a new construction based on
Artin-Schreier theory that allows to do most arithmetic operations in
the tower in quasi-linear time. Assuming this construction is used, we
can now give precise bounds for each step of C2-AS.

\paragraph{$p$-torsion}
The construction of $\F_q[c]$ may be done in many ways. The only
requirements of \cite{DFS09} are
\begin{enumerate}
\item that its elements have a representation as elements of
  $F_p[X]/Q_1(X)$ for some irreducible polynomial $Q_1$,
\item that either $(d,p)=1$ or $\deg Q_1' + 2 = \deg Q_1$.
\end{enumerate}
Selecting a random polynomial $Q_1$ and testing for irreducibility is
usually enough to meet these conditions. This costs
$O\bigl(pd\Mult(pd)\log (pd)\log(p^2d)\bigr)$ according to
\cite[Th. 14.42]{vzGG}.

Now we need to compute the embedding $\F_q\subset\F_q[c]$. Supposing
$\F_q$ is represented as $\F_p[X]/Q_0(X)$, we factor $Q_0$ in
$\F_q[c]$, which costs $O\bigl(pd\Mult(pd^2)\log d\log p\bigr)$ using
\cite[Coro. 14.16]{vzGG}. Then the most naive technique to express the
embedding is linear algebra. This requires the computation of $pd$
elements of $\F_q[c]$ at the expense of $\Theta\bigl(pd\Mult(pd)\bigr)$
operations in $\F_p$, then the inversion of the matrix holding such
elements, at a cost of $\Theta\bigl((pd)^\omega\bigr)$ operations. This is
certainly not optimal, yet this phase will have negligible cost
compared to the rest of the algorithm.

Now we can compute $c$ and $c'$ by factoring the polynomials
$Y^{p-1}-H_E$ and $Y^{p-1}-H_{E'}$ in $\F_p[X]/Q_1(X)$. This costs
\[O\bigl((p\ModComp(pd) + \ModComp(p)\Mult(pd) + \Mult(p)\Mult(pd)\log
p)(\log^2 p+\log d)\bigr)\]
using \cite[Section 3]{KS97}.

Finally, computing the determinants needed by Gunji's formulae takes
$\Theta(p^2)$ multiplications in $\F_q[c]$, that is
$\Theta\bigl(p^2\Mult(pd)\bigr)$.

Letting out logarithmic factors, the overall cost of this phase is
\begin{equation}
  \label{eq:gunji-complexity}
  \tildO\bigl(p^2d^3 + p\ModComp(pd) + \ModComp(p)pd + (pd)^\omega \bigr)
\end{equation}

\paragraph{$p^k$-torsion}
Application of Voloch formulae requires at each of the levels
$\U_2,\ldots,\U_k$
\begin{enumerate}
\item to solve equation \eqref{th:voloch:cover} by factoring an
  Artin-Schreier polynomial,
\item to solve the system \eqref{th:voloch:isom}.
\end{enumerate}
If we assume the worst case $[\U_2:\U_1] = p$, according to
\cite[Th. 13]{DFS09}, at each level $i$ the first step costs
\begin{gather*}
  O\bigl((pd)^\omega i + {\sf PT}(i-1) + \Mult(p^{i+1}d)\log p\bigr)\\[0.3cm]
  \begin{aligned}
    \text{where}&&
    {\sf PT}(i) &= O\bigl((pi + \log(d))i{\sf L}(i) + p^i\ModComp(pd)\log^2(pd)\bigr)\\
    \text{and}&&
    {\sf L}(i) &= O\bigl(p^{i+2}d\log_p^2{p^{i+1}d} + p\Mult(p^{i+1}d)\bigr)
    \text{ ;}
  \end{aligned}
\end{gather*}
while the second takes the GCD of two degree $p$ polynomials in
$\U_i[X]$ for each $i$ (see Section \ref{sec:implementation}), at a
cost of $O\bigl(\Mult(p^{i+1}d)\log p\bigr)$ operations using a fast
algorithm \cite[$\S$11.1]{vzGG}.

Summing up over $i$, the total cost of this phase up to logarithmic
factors is
\begin{equation}
  \label{eq:C2-AS:complexity:p^k}
  \tildO_{p,d,\log\ell}\left((pd)^\omega \log_p^2\ell + p^2\ell d\log_p^4\ell +
  \frac{\ell}{p}\ModComp(pd)\right)
  \;\text{.}  
\end{equation}
Also notice that there is no more need to store a $p^{k-1}d\times
p^{k-1}d$ matrix to solve the Artin-Schreier equation, thus the space
requirements are not anymore quadratic in $\ell$.

\paragraph{Interpolation}
The interpolation phase is not essentially changed: one needs first to
interpolate a degree $p^k-1$ polynomial with coefficients in $\U_k$,
then use \cite[\alg{Push-down}]{DFS09} to obtain the corresponding
polynomial in $\F_q[X]$ and finally do a rational fraction
reconstruction.

The first step costs $O\bigl(\Mult(p^{2k}d)\log p^k\bigr)$ using fast techniques
as \cite[$\S$10.2]{vzGG}, then converting to $\F_q[c][X]$ takes
$O\bigl(p^k{\sf L}(k-1)\bigr)$ by \cite{DFS09} and further converting to
$\F_q[X]$ takes $\Theta\bigl((pd)^2\bigr)$ by linear algebra. The rational function
reconstruction then takes $O\bigl(\Mult(p^kd)\log p^k\bigr)$ using
fast GCD techniques \cite[$\S$11.1]{vzGG}.

The overall complexity of one interpolation is then
\begin{equation}
  \label{eq:C2-AS:complexity:interp}
  O\bigl(\Mult(\ell^2d)\log_p\ell + \ell{\sf L}(k-1) + (pd)^2\bigr)
  \;\text{.}
\end{equation}
Remember that this step has to be repeated an average number of
$\euler(p^k)/4$ times, thus the dependency of C2-AS in $\ell$ is still cubic.

%

%% file: C2-AS-FI.tex
\section{The algortihm C2-AS-FI}
\label{sec:C2-AS-FI}

The most expensive step of C2-AS is the polynomial interpolation step
which is part of the Cauchy interpolation. If we use a standard
interpolation algorithm, its input consists in a list of $\Theta(p^k)$
pairs $\bigl(P, \I(P)\bigr)$ with $P\in\U_k$, thus a lower bound for
any such algorithm is $\Omega(p^{2k}d)$. Notice however that the
output is a polynomial of degree $\Theta(p^k)$ in $\F_q[X]$, hence, if
supplied with a shorter input, an \emph{ad hoc} algorithm could reach
the bound $\Omega(p^kd)$.

In this section we give an algorithm that reaches this bound up to
some logarithmic factors. It realizes the polynomial interpolation on
the primitive points of $E[p^k]$, thus its output is a degree
$\euler(p^k)/2-1$ polynomial in $\F_q[X]$. Using the Chinese remainder
theorem, it is straightforward to generalise this to an algorithm
having the same asymptotic complexity realizing the polynomial
interpolation on all the points of $E[p^k]$. We call C2-AS-FI the
variant of C2-AS resulting from applying this new algorithm.

\subsection{The algorithm}
We set some notation. Let $i_0$ be the largest index such that
$\U_{i_0} = \U_1$ and let $\frac{p-1}{2r} = [\F_q[c^2]:\F_q]$. For
notational convenience, we set $\U_0=\F_q$.

We note $T(X)$ the polynomial vanishing on the primitive points of
$E[p^k]$ and

\begin{equation}
  \label{eq:T}
  T = \prod T_j^{(i)}
\end{equation}
its factorisation over $\U_i$; we remark that all the $T_j^{(0)}$'s
have degree $\frac{\euler(p^{k-i_0+1})}{2r}$. We also note $A(X)$ the goal
polynomial and
\begin{equation}
  \label{eq:A}
  A_j^{(i)} = A \bmod T_j^{(i)}
  \;\text{.}
\end{equation}

It was already pointed out in \cite[$\S$2.3]{Cou96} that if all the
$A_j^{(0)}$'s are known one can recover $A$ using the Chinese remainder
theorem. If we chose any point $P$ such that
$T_j^{(0)}\bigl(x(P)\bigr)=0$ and fix the embedding
\begin{equation}
  \label{eq:embed}
  \xymatrix{
    ^{\F_q[X]}/_{T_j^{(0)}(X)} \ar@{^{(}->}[r]^-\iota & \U_k
  }
\end{equation}
given by $\iota(X) = x(P)$, then it is evident that
$\iota\bigl(A_j^{(0)}(X)\bigr) = x\bigl(\I(P)\bigr)$, thus in order to
compute $A_j^{(0)}$ one just needs to compute
$\iota^{-1}\bigl(x(\I(P))\bigr)$.

Unfortunately, the information needed to compute $\iota$ was lost in
the $p$-descent, for we don't even know the $T_j^{(i)}$'s. None of the
algorithms of \cite{DFS09} helps us to compute such information and
straightforward computation of it would be too expensive. The solution
is to decompose $\iota$ as a chain of morphisms and invert them
one-by-one going down in the tower $(\U_0,\U_1,\ldots,\U_k)$, this is
similar to the way \cite{Cou00} solves an Artin-Schreier equation by
moving it down from $\U_k$ to $\U_1$.

\paragraph{The moduli}
We first need to compute $T_0^{(i)}\in\U_i[X]$ for any $i$. For this
we fix a primitive point $P\in E[p^k]$ and we reorder the indices in
\eqref{eq:T} so that $T_0^{(i)}$ is the minimal polynomial of $x(P)$
over $\U_i$.

The first minimal polynomial is simply
\begin{equation}
  \label{eq:T_0^k}
  T_0^{(k)}(X) = X - x(P)
  \;\text{.}
\end{equation}
Now suppose we know $T_0^{(i+1)}$, then a generator $\sigma$ of
$\Gal(\U_{i+1}/\U_i)$ acts on the roots of $T_0^{(i+1)}$ sending them
on the roots of some $T_j^{(i+1)}$. Then the
minimal polynomial of $x(P)$ over $\U_i$ is
\begin{equation}
  \label{eq:T_0^i}
  T_0^{(i)} = \prod_{\sigma\in\Gal(\U_{i+1}/\U_i)} \sigma\left(T_0^{(i+1)}\right)
  \;\text{.}
\end{equation}
Some care has to be taken when computing $T_0^{(0)}$: in fact the
abscissae of the points may be counted twice if
$c\not\in\F_q[c^2]$. In this case only a subgroup of index $2$ of
$\Gal(\U_1/\U_0)$ must be used instead of the whole group.

\paragraph{The interpolation}
The computation of $A_0^{(i)}$ is done in the same recursive way. Fix
the same point $P$ used to compute the $T_0^{(i)}$'s and fix the chain
of embeddings
\begin{equation}
  \xymatrix{
    ^{\U_0[X_0]}/_{T_0^{(0)}(X_0)} \ar@{^{(}->}[r]^-{\iota_0} &
    \;\cdots\; \ar@{^{(}->}[r]^-{\iota_{k-1}} &
    ^{\U_k[X_k]}/_{T_0^{(k)}(X_k)} \ar@{^{(}->}[r]^-{\iota_k} &
    \U_k
  }
\end{equation}
given by $\iota_k\circ\cdots\circ\iota_i(X_i) = x(P)$ for any $i$.

We compute $A_0^{(i)}$ by inverting the chain: inverting $\iota_k$
simply gives
\begin{equation}
  \label{eq:A_0^k}
  A_0^{(k)} = x\bigl(\I(P)\bigr)
  \;\text{.}
\end{equation}
Then suppose we know $A_0^{(i+1)}$, and decompose the embedding
$\iota_i$ as
\begin{equation}
  \xymatrix{
    ^{\U_i[X_i]}/_{T_0^{(i)}(X_i)} \ar@{^{(}->}[r]^-{\iota_i} \ar@{^{(}->}[d]^{\varepsilon} &
    ^{\U_{i+1}[X_{i+1}]}/_{T_0^{(i+1)}(X_{i+1})} \\
    ^{\U_{i+1}[Y]}/_{T_0^{(i)}(Y)} \ar@{^{(}->>}[r]^-{\gamma} &
    \bigoplus_j {}^{\U_{i+1}[Y_{j}]}/_{T_j^{(i+1)}(Y_{j})} \ar@{->>}[u]_{\pi}
  }
\end{equation}
where $\varepsilon$ is the canonical injection extending
$\U_i\subset\U_{i+1}$, $\gamma$ is the Chinese remainder isomorphism
and $\pi$ is projection onto the first coordinate.

To invert $\pi$ observe that any $\sigma\in\Gal(\U_{i+1}/\U_i)$ leaves
$A_0^{(i)}$ invariant while it permutes the moduli $T_j^{(i+1)}$, thus
\begin{equation}
  A_0^{(i)} \equiv \sigma\left(A_0^{(i+1)}\right)
  \bmod \sigma\left(T_0^{(i+1)}\right)
  \;\text{;}
\end{equation}
Hence we can obtain all the $A_j^{(i+1)}$ through the action of
$\Gal(\U_{i+1}/\U_i)$ on $A_0^{(i+1)}$.

Then we can invert $\gamma$ through a Chinese remainder algorithm
\cite[$\S$10.3]{vzGG} and $\varepsilon$ by converting coefficients from
$\U_{i+1}$ to $\U_i$.

As for the moduli, a special treatment is needed for $\iota_0$ if
$c\not\in\F_q[c^2]$.

\subsection{Complexity analysis}
\label{sec:C2-AS-FI:complexity}

The two algorithms for computing the $T_{0}^{(i)}$'s and the
$A_{0}^{(i)}$'s are very similar and run in parallel. We can merge
them in one unique algorithm, at each level $i\ge i_0$ it does the
following

\begin{enumerate}
\item for $\sigma \in \Gal(\U_{i+1}/\U_i)$, call $\bar\sigma$ the
  permutation it induces on the indices of the $T_j^{(i+1)}$'s,
  compute
  \begin{enumerate}
  \item\label{alg:T:gal} $T_{\bar\sigma(0)}^{(i+1)} :=
    \sigma\left(T_0^{(i+1)}\right)$ and
  \item\label{alg:A:gal} $A_{\bar\sigma(0)}^{(i+1)} :=
    \sigma\left(A_0^{(i+1)}\right)$ using
    \cite[\alg{IterFrobenius}]{DFS09},
  \end{enumerate}
\item\label{alg:T:prod} compute $T_0^{(i)}$ through a subproduct tree
  as in \cite[Algo. 10.3]{vzGG},
\item\label{alg:A:CRA} compute $A_0^{(i)}$ through Chinese Reminder
  Algorithm \cite[Algo. 10.16]{vzGG},
\item\label{alg:T:push} convert $T_0^{(i)}$ and $A_0^{(i)}$ into
  elements of $\U_i[X]$ using \cite[\alg{Push-down}]{DFS09}.
\end{enumerate}

Steps \ref{alg:T:gal} and \ref{alg:A:gal} are identical. Both are
repeated $p$ times, each iteration taking $O\bigl(p^{k-i}{\sf
  L}(i-i_0)\bigr) \subset O\bigl({\sf L}(k-i_0)\bigr)$ by
\cite[Th. 17]{DFS09}.

Step \ref{alg:T:prod} takes $O\bigl(\Mult(p^{k-i_0+1}d/r)\log p\bigr)$
by \cite[Lemma 10.4]{vzGG} and step \ref{alg:A:CRA} has the same
complexity by \cite[Coro. 10.17]{vzGG}.

Step \ref{alg:T:push} takes $O\bigl(p^{k-i+1}{\sf L}(i-i_0)\bigr)
\subset O\bigl(p{\sf L}(k-i_0)\bigr)$.

When $i=0$ and $\U_1\ne\F_q$ the algorithm is identical but steps
\ref{alg:T:gal} and \ref{alg:A:gal} must be computed through a generic
frobenius algorithm (using~\cite[Algorithm 5.2]{vzGS92}, for example)
and step \ref{alg:T:push} must use the implementation of $F_q[c]$ to
make the conversion (for example, linear algebra). In this case steps
\ref{alg:T:gal} and \ref{alg:A:gal} cost
$\Theta\bigl(\frac{p^{k-i_0}}{r}\ModComp(pd)\log d \bigr)$
by~\cite[Lemma 5.3]{vzGS92} and step \ref{alg:T:push} costs
$\Theta\bigl(p^{k-i_0}(pd)^2\bigr)$.

The total cost of the algorithm is then
\begin{equation*}
  \label{eq:T:complexity}
  O\left(\bigl(k-i_0\bigr)\bigl(p{\sf L}(k-i_0) + \Mult(p^{k-i_0+1}d/r)\log p\bigr) +
    \frac{p^{k-i_0}}{r}\bigl(\ModComp(pd)\log d + r(pd)^2\bigr) \right)
  \;\text{.}
\end{equation*}

After all, the whole algorithm looks a lot like fast interpolation
\cite[$\S$10]{vzGG} and it is indeed a modified version of it. A
similar algorithm was already given in \cite{EnMo03}.

\paragraph{The complete interpolation}
We compute all the $A_j^{(0)}$'s using this algorithm; there's
$p^{i_0-1}r$ of them. We then recombine them through a Chinese
remainder algorithm at a cost of $O\bigl(\Mult(p^kd)\log
p^{i_0-1}r\bigr)$. The total cost of the whole interpolation phase is
then
\begin{equation*}
  O\left(\bigl(k-i_0\bigr) \bigl(p{\sf L}(k) + \Mult(p^kd)\log p\bigr) +
    p^{k-1}\ModComp(pd)\log d + p^{k-1}r(pd)^2 + i_0\Mult(p^kd)\log p
  \right)
  \;\text{,}
\end{equation*}
that is
\begin{equation}
  \label{eq:interp}
  O\left(p{\sf L}(k)\log\left(\frac{\ell}{p^{i_0}}\right) + 
    \Mult(\ell d)\log\ell\log p +
    \frac{\ell}{p}\ModComp(pd)\log d +
    \ell (pd)^2
  \right)
  \;\text{.}
\end{equation}

Alternatively, once $A_0^{(0)}$ is known, one could compute the other
$A_j^{(0)}$'s using modular composition with the multiplication maps
of $E$ and $E'$ as suggested in \cite{Cou96}. However this approach
doesn't give a better asymptotic complexity because in the worst case
$A_0^{(0)}=A$. From a practical point of view, though, Brent's and
Kung's algorithm for modular composition \cite{BrKu78}, despite having
a worse asymptotic complexity, could perform faster for some set of
parameters. We will discuss this matter in Section
\ref{sec:C2-AS-FI-MC}.

If more than $\euler(p^k)/2$ points are needed, but less than
$\frac{p-1}{2}$, one can use the previous algorithm to compute all the
polynomials $A_i$ interpolating respectively over the $p^i$-torsion
points of $E$ and $E'$. They can then be recombined through a Chinese
remainder algorithm at a cost of $O\bigl(\Mult(p^kd)\log p^k\bigr)$,
which doesn't change the overall complexity of C2-AS-FI.

Putting together the complexity estimates of C2-AS and C2-AS-FI, we
have the following.

\begin{theorem}
  \label{th:complexity}
  Assuming $\Mult(n) = n\log n\log\log n$, the algorithm C2-AS-FI has
  worst case complexity
  \begin{equation*}
    \tildO_{p,d,\log\ell}\left(
      p^2d^3 +
      \ModComp(p)pd +
      (pd)^\omega\log^2\ell +
      p^3\ell^2 d\log^3\ell + 
      p^2\ell^2 d^2+
      \left(\frac{\ell^2}{p} + p\right)\ModComp(pd)
    \right)
    \;\text{.}
  \end{equation*}
\end{theorem}

%

%% file: C2-AS-FI-MC.tex
\section{The algorithm C2-AS-FI-MC}
\label{sec:C2-AS-FI-MC}

However asymptotically fast, the polynomial interpolation step is
quite expensive for reasonably sized data. Instead of repeating it
$\frac{\euler(p^k)}{2}$ times, one can use composition with the
Frobenius endomorphism $\frobisog_E$ in order to reduce the number of
interpolations in the final loop.

\subsection{The algorithm}
Suppose we have computed, by the algorithm of the previous Section,
the polynomial $T$ vanishing on the abscissae of $E[p^k]$ and an
interpolating polynomial $A_0\in\F_q[X]$ such that
\begin{equation*}
  A_0\bigl(x\bigl([n]P\bigr)\bigr) = x\bigl([n]P'\bigr)
  \quad\text{for any $n$.}
\end{equation*}
The group $\Gal(\U_k/\F_q) = \langle\frob\rangle$ acts on $E'[p^k]$
permuting its points and preserving the group structure. Thus, the
polynomial
\begin{equation*}
  A_1 = A_0\circ\frob = \frob\circ A_0
\end{equation*}
is an interpolating polynomial such that
\begin{equation*}
  A_1\bigl(x\bigl([n]P\bigr)\bigr) = x\bigl([n]\frobisog_{E'}(P')\bigr)
  \quad\text{for any $n$,}
\end{equation*}
where $\frobisog_{E'}$ is the Frobenius endomorphism of $E'$.  Since
$\frobisog_{E'}(P')$ is a generator of $E'[p^k]$, $A_1$ is one of the
polynomials that the algorithm C2 tries to identify to an isogeny. By
iterating this construction we obtain $[\U_k:\F_q]/2$ different
polynomials $A_i$ for the algorithm C2 with only interpolation.

To compute the $A_i$'s, we first compute $F\in\F_q[X]$
\begin{equation}
  \label{eq:frob}
  F(X) = X^q \bmod T(X)
  \text{,}
\end{equation}
then for any $1\le i<[\U_k:\F_q]/2$
\begin{equation}
  \label{eq:modcomp}
  A_i(X) = A_{i-1}(X)\circ F(X) \bmod T(X)\text{.}
\end{equation}

If $\frac{\euler(p^k)}{[\U_k:\F_q]} = p^{i_0-1}r$, we must compute
$p^{i_0-1}r$ polynomial interpolations and apply this algorithm to
each of them in order to deduce all the polynomials needed by C2.

\subsection{Complexity analysis}
We compute \eqref{eq:frob} via square-and-multiply, this costs
$\Theta(d\Mult(p^kd)\log p)$ operations. Each application of
\eqref{eq:modcomp} is done via a \emph{modular composition}, the cost
is thus $O(\ModComp(p^k))$ operations in $\F_q$, that is
$O(\ModComp(p^k)\Mult(d))$ operations in $\F_p$. Using the algorithm
of~\cite{KeUm08} for modular composition, the complexity of
C2-AS-FI-MC wouldn't be essentially different from the one of
C2-AS-FI; however, in practice the fastest algorithm for modular
composition is~\cite{BrKu78}, and in particular the variant
in~\cite[Lemma 3]{KS98}, which has a worse asymptotic complexity, but
performs better on the instances we treat in
Section~\ref{sec:benchmarks}.

Notice that a similar approach could be used inside the polynomial
interpolation step (see Section \ref{sec:C2-AS-FI}) to deduce
$A_k^{(0)}$ from $A_0^{(0)}$ using modular composition with the
multiplication maps of $E$ and $E'$ as described in
\cite[$\S$2.3]{Cou96}. This variant, though, has an even worse
complexity because of the cost of computing multiplication maps.

%

%% file: implementation.tex
\section{Implementation}
\label{sec:implementation}

We implemented C2-AS-FI-MC as \texttt{C++} programs using the
libraries \texttt{NTL}~\cite{NTL} for finite field arithmetics,
\texttt{gf2x}~\cite{gf2x} for fast arithmetics in characteristic $2$
and \texttt{FAAST}~\cite{DFS09} for fast arithmetics in Artin-Schreier
towers.

This section mainly deals with some tricks we implemented in order to
speed up the computation. At the end of the section we briefly discuss
the implementation we made in Magma~\cite{Magma} of the algorithm
in~\cite{LeSi09}.

\subsection{Building $E[p^k]$ and $E'[p^k]$}
\label{sec:impl:torsion}

\paragraph{$p$-torsion}
For $p\ne2$, C2 and its variants require to build the extension
$\F_q[c]$ where $c$ is a $p-1$-th root of $H_E$. In order to deal with
the lowest possible extension degree, it is a good idea to modify the
curve so that $[\F_q[c]:\F_q]$ is the smallest possible.

$[\F_q[c]:\F_q]$ is invariant under isomorphism, but taking a twist
can save us a quadratic extension. Let $u=c^{-2}$, the curve
\begin{equation*}
  \bar{E} : y^2 = x^3 + a_2ux^2 + a_4u^2x + a_6u^3
\end{equation*}
is defined over $\F_q[c^2]$ and is isomorphic to $E$ over $\F_q[c]$
via $(x,y)\mapsto(\sqrt{u}^2x,\sqrt{u}^3y)$. Its Hasse invariant is
$H_{\bar{E}} = (u)^{\frac{p-1}{2}}H_E = 1$, thus its $p$-torsion
points are defined over $\F_q[c^2]$.

In order to compute the $p^k$-torsion points of $E$ we build
$\F_q[c^2]$, we compute $\bar{P}$ a $p^k$-torsion points of $\bar{E}$
using $p$-descent, then we invert the isomorphism to compute the
abscissa of $P\in E[p^k]$. Since the Cauchy interpolation only needs
the abscissae of $E[p^k]$, this is enough to complete the
algorithm. Scalar multiples of $P$ can be computed without knowledge
of $y(P)$ using Montgomery formulae \cite{Mon87}.

Remark that for $p=2$ we use the same construction in an implicit way
since we do a $p$-descent on the Kummer surface.

\paragraph{$p^k$-torsion points}
For $p\ne2$ we use Voloch's $p$-descent to compute the $p^k$-torsion
points iteratively as described in Section \ref{sec:C2}. To factor the
Artin-Schreier polynomial \eqref{th:voloch:cover}, we use the
algorithms from \cite{Cou00} and \cite{DFS09} that were analysed in
Section \ref{sec:C2-AS}. All these algorithms were provided by the
library \texttt{FAAST}.

To solve system \eqref{th:voloch:isom} we first compute
\begin{equation*}
  V(x,y) = \left(\frac{g(x)}{h^2(x)}, 
    sy\left(\frac{g(x)}{h^2(x)}\right)'\right)
\end{equation*}
through Vélu formulae.\footnote{Vélu formulae compute this isogeny up
  to an indeterminacy on the sign of the ordinate, the actual value of
  $s$ must be determined by composing $V$ with $\frobisog$ and
  verifying that it corresponds to $[p]$ by trying some random
  points.} Recall that we work on a curve having Hasse invariant $1$,
system \eqref{th:voloch:isom} can then be rewritten
\begin{equation*}
  \left\{
    \begin{aligned}
      X &= \frac{g(x)}{h^2(x)}\\
      Y &= sy\left(\frac{g(x)}{h^2(x)}\right)'\\
      Z &= -2y\frac{h'(x)}{h(x)}
    \end{aligned}
  \right.
\end{equation*}
where $(X,Y,Z)$ is the point on the cover $C$ that we want to pull
back. After some substitutions this is equivalent to
\begin{equation*}
  \left\{
    \begin{aligned}
      Xh^2(x) - g(x) &= 0\\
      \left(Xh^2(x) - g(x) - \frac{Y}{sZ}h^2(x)\right)' &= 0
    \end{aligned}
  \right.
\end{equation*}
Then a solution to this system is given by the GCD of the two
equations. Remark that proposition \ref{th:voloch} ensures there is
one unique solution. This formulae are slightly more efficient than
the ones in \cite[$\S$6.2]{Ler97}.

For $p=2$ we use the library \texttt{FAAST} (for solving
Artin-Schreier equations) on top of \texttt{gf2x} (for better
performance). There is nothing special to remark about the
$2$-descent.

\subsection{Cauchy interpolation and loop}
\label{sec:impl:cauchy}
The polynomial interpolation step is done as described in Section
\ref{sec:C2-AS-FI}. As a result of this implementation, the polynomial
interpolation algorithm was added to the library \texttt{FAAST}.

The rational fraction reconstruction is implemented using a fast XGCD
algorithm on top of \texttt{NTL} and \texttt{gf2x}. This algorithm was
added to \texttt{FAAST} too.

The loop uses modular composition as in Section~\ref{sec:C2-AS-FI-MC}
in order to minimise the number of interpolations. The timings in the
next section clearly show that this non-asymptotically-optimal variant
performs much faster in practice.

To check that the rational fractions are isogenies we test their
degrees, that their denominator is a square and that they act as group
morphisms on a fixed number of random points. All these checks take a
negligible amount of time compared to the rest of the algorithm.

\subsection{Parallelisation of the loop}
\label{parallel}

The most expensive step of C2-AS-FI-MC, in theory as well as in
practice, is the final loop over the points of $E'[p^k]$. Fortunately,
this phase is very easy to parallelise with very few overhead.

Let $n$ be the number of processors we wish to parallelise on, suppose
that $[\U_k:\F_q]$ is maximal, then we make only one interpolation
followed by $\euler(p^k)/2$ modular compositions.\footnote{If
  $[\U_k:\F_q]$ is not maximal, the parallelisation is straightforward
  as we simply send one interpolation to each processor in turn.} We
set $m=\left\lfloor\frac{\euler(p^k)}{2n}\right\rfloor$ and we compute
the action of $\frob^{m}$ on $E[p^k]$ as in
Section~\ref{sec:C2-AS-FI-MC}:
\begin{equation*}
  F^{(m)}(X) = F(X) \circ \cdots \circ F(X) \bmod T(X)
  \;\text{,}
\end{equation*}
this can be done with $\Theta(\log m)$ modular compositions via a
binary square-and-multiply approach as in~\cite[Algorithm
5.2]{vzGS92}.

Then we compute the $n$ polynomials
\begin{equation*}
  A_{mi}(X) = A_{m(i-1)}(X) \circ F^{(m)}(X) \bmod T(X)
\end{equation*}
and distribute them to the $n$ processors so that they each work on a
separate slice of the $A_i$'s. The only overhead is $\Theta(\log
(\ell/n))$ modular compositions with coefficients in $\F_q$, this is
acceptable in most cases.

\subsection{Implementation of~\cite{LeSi09}}
In order to compare our implementation with the state-of-the-art
algorithms, we implemented a Magma prototype of~\cite{LeSi09}; in what
follows, we will refer to this algorithm as LS. The algorithm
generalises~\cite{BoMoSaSc08} by lifting the curves in the $p$-adics
to avoid divisions by zero. Given two curves $E$ and $E'$ and an
integer $\ell$, it performs the following steps
\begin{enumerate}
\item Lift $E$ to $\bar{E}$ in $\Q_q$,
\item Lift the modular polynomial $\Phi_\ell$ to $\bar{\Phi}_\ell$ in
  $\Q_q$,
\item Find a root in $\Q_q$ of $\bar{\Phi}(X,j_{\bar{E}})$ that
  reduces to $j_{E'}$ in $\F_q$,
\item Apply~\cite{BoMoSaSc08} in $\Q_q$ to find an isogeny between
  $\bar{E}$ and $\bar{E}'$,
\item Reduce the isogeny to $\F_q$.
\end{enumerate}

We implemented this algorithm using Magma support for the
$p$-adics. Instead of the classical modular polynomials $\Phi_\ell$ we
used Atkin's canonical polynomials $\Phi^\ast_\ell$ since they have
smaller coefficients and degree; this does not change the other steps
of the algorithm. The modular polynomials were taken from the tables
precomputed in Magma.

The bottleneck of the algorithm is the use of the modular polynomial
as its bit size is $O(\ell^3)$, thus LS is asymptotically worse in
$\ell$ than C2. However the next section will show that LS is more
practical than C2 in many circumstances.

%

%% file: benchmarks.tex
\section{Benchmarks}
\label{sec:benchmarks}
We ran various experiments to compare the different variants of the
algorithm C2 between themselves and to the other algorithms. All the
experiments were run on four dual-core Intel Xeon E5430 (2.6GHz),
eventually using the parallelised version of the algorithm.

\begin{figure}
  \centering
  \includegraphics[width=0.9\textwidth]{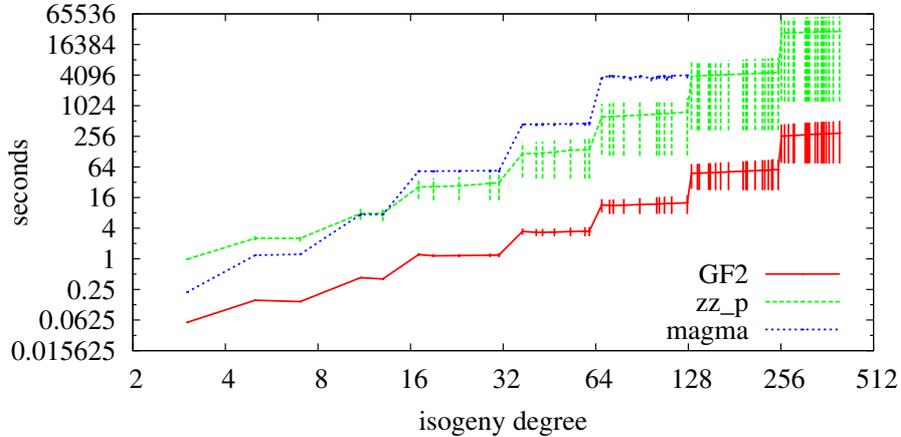}
  \caption{Comparative timings for different implementations of C2-AS-FI-MC with curves defined over $\F_{2^{101}}$. Plot in logarithmic scale.}
  \label{fig:2-101}
\end{figure}

The first set of experiments was run to evaluate the benefits of using
the fast algorithms in~\cite{DFS09}. We selected pairs of isogenous
curves over $\F_{2^{101}}$ such that the height of the tower is
maximal (observe that this is always the case for cryptographic
curves). The library \texttt{FAAST} offers two types for finite field
arithmetics in characteristic $2$: \texttt{zz\_p} which is a generic
type for word-precision $p$ and \texttt{GF2} which uses the optimised
algorithms of the library \texttt{gf2x}. We compared implementations
of C2-AS-FI-MC using these two types with an implementation written in
Magma. The results are in figure~\ref{fig:2-101}: we plot a line for
the average running time of the algorithm and bars around it for
minimum and maximum execution times of the final loop. Besides the
dramatic speedup obtained by using the ad-hoc type \texttt{GF2}, the
algorithmic improvements of \texttt{FAAST} over Magma are evident as
even \texttt{zz\_p} is one order of magnitude faster.

\begin{table}
  \centering
  \begin{tabular}{r r r r r r r r}
    \hline
    $\ell$ & $E[p^k]$ & $E'[p^k]$ & FI & RFR & MC & Avg tries & Avg loop time\\
    \hline
    31 & 1.3128 & 1.3128 & 1.1058 & 0.00218 & 0.00218 & 64 & 0.279\\
    61 & 3.5454 & 3.5464 & 2.5236 & 0.00783 & 0.00900 & 128 & 2.154 \\
    127 & 9.2975 & 9.3026 & 5.6881 & 0.03147 & 0.03634 & 256 & 17.359 \\
    251	& 23.7984 & 23.7984 & 12.7251 & 0.12415 & 0.14519 & 512 & 137.902 \\
    397 & 59.7439 & 59.7579 & 28.3387 & 0.36822 & 0.58027 & 1024 & 971.254 \\
    \hline
  \end{tabular}
  \caption{Comparative timings for the phases of C2-AS-FI-MC for curves over $\F_{2^{101}}$.}
  \label{tab:C2}
\end{table}

Table~\ref{tab:C2} shows detailed timings for each phase of
C2-AS-FI-MC. The column FI reports the time for one interpolation, the
column MC the time for one modular composition; comparing these two
columns the gain from passing from C2-AS-FI to C2-AS-FI-MC is
evident. Columns RFR (rational fraction reconstruction) and MC
constitute the Cauchy interpolation step that is repeated in the final
loop. The last column reports the average time spent in the loop: it
is by far the most expensive phase and this justifies the attention we
paid to FI and MC; only on some huge examples we approached the
crosspoint between these two algorithms.

\begin{figure}
  \centering
  \includegraphics[height=0.45\textwidth]{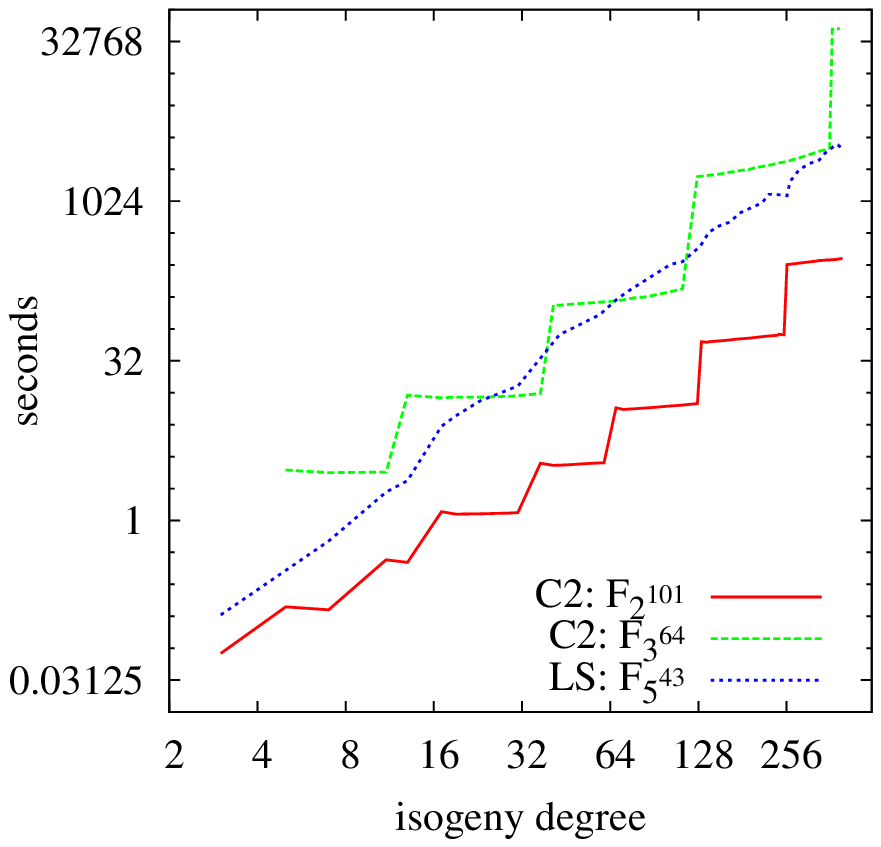}
   \includegraphics[height=0.45\textwidth]{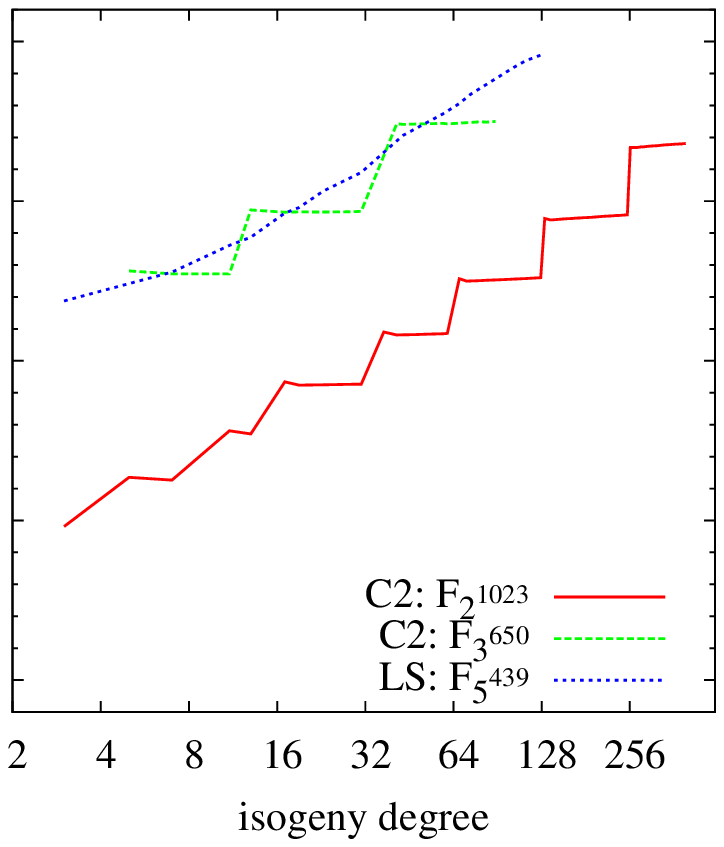}
   \caption{Comparative timings for C2-AS-FI-MC (C2) and LS over
     different curves. Plot in logarithmic scale.}
  \label{fig:comp}
\end{figure}

Next, we compare the running times of C2-AS-FI-MC and LS over curves
of half the cryptographic size in figure~\ref{fig:comp} (left). We
only plot average times for C2, in characteristic $2$ we only plot the
timings for \texttt{GF2}. From the plot it is clear that C2-AS-FI-MC
only performs better than LS for $p=2$, but in this case the algorithm
of~\cite{Ler96} is by far better.  Figure~\ref{fig:comp} (right) shows
that LS slowly gets worse than C2, however comparing a Magma prototype
to our highly optimised implementation of C2-AS-FI-MC is somewhat
unfair and probably the crosspoint between the two algorithms lies
much further. Furthermore, it is unlikely that C2-AS-FI-MC could be
practical for any $p>3$ because of its high dependence on $p$, while
LS scales pretty well with the characteristic as shown in
figure~\ref{fig:LSp}.

\begin{figure}
  \centering
  \includegraphics[width=0.9\textwidth]{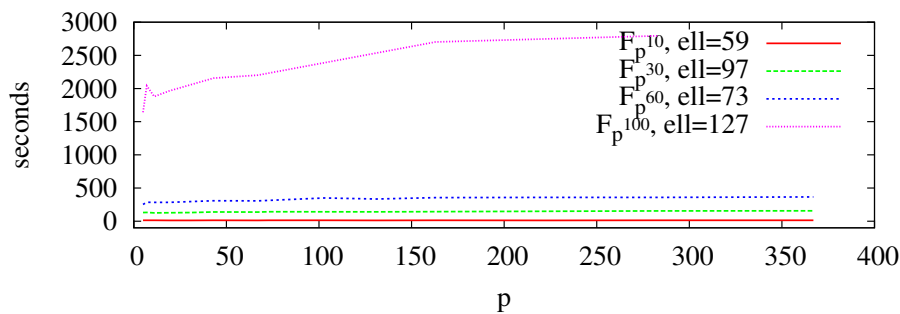}
  \caption{Timings for LS for different fields. We increase
    $p$ while taking constant $d$ and the isogeny degree.}
  \label{fig:LSp}
\end{figure}

We can hardly hide our disappointment concluding that, despite their
good asymptotic behaviour and our hard work implementing them, the
variants derived from C2 don't seem to have any practical application,
at least for present data sizes. We hope that in the future the
algorithms presented here may turn useful to compute very large data
that are currently out of reach.

%